\documentclass{amsart} %[11pt, a4paper, twoside, onecolumn]

\usepackage{amsmath,amssymb,amsthm,amsfonts,amsxtra,mathtools}
\usepackage{enumerate,verbatim,mathrsfs,comment,color,url}
\usepackage[all,2cell,ps]{xy}
\usepackage[pagebackref]{hyperref}
\usepackage{graphicx}
\usepackage{tikz-cd}
\usepackage{cleveref}
\usepackage{verbatim}

\DeclareMathOperator{\ann}{ann}

\DeclareMathOperator{\cid}{CI-dim}
\DeclareMathOperator{\cidim}{CI-dim}

\DeclareMathOperator{\coker}{Coker}

\DeclareMathOperator{\depth}{depth}

\DeclareMathOperator{\Ext}{Ext}

\DeclareMathOperator{\gdim}{G-dim}

\DeclareMathOperator{\Hom}{Hom}

\DeclareMathOperator{\id}{id}
\DeclareMathOperator{\Image}{Image}

\DeclareMathOperator{\Ker}{Ker}
\DeclareMathOperator{\Max}{MaxSpec}
\DeclareMathOperator{\MaxSpec}{MaxSpec}

\DeclareMathOperator{\pd}{pd}

\DeclareMathOperator{\Tor}{Tor}

\renewcommand{\ge}{\geqslant}
\renewcommand{\le}{\leqslant}

\newcommand{\fm}{\mathfrak{m}}

\renewcommand{\iff}{if and only if }

\theoremstyle{plain}
\newtheorem{theorem}{Theorem}[section]
\newtheorem{lemma}[theorem]{Lemma}
\newtheorem{proposition}[theorem]{Proposition}
\newtheorem{corollary}[theorem]{Corollary}

\newtheorem{innercustomtheorem}{Theorem}
\newenvironment{customtheorem}[1]
{\renewcommand\theinnercustomtheorem{#1}\innercustomtheorem}
{\endinnercustomtheorem}

\theoremstyle{definition}

\newtheorem{conjecture}[theorem]{Conjecture}

\newtheorem{example}[theorem]{Example}

\newtheorem{para}[theorem]{}

\theoremstyle{remark}
\newtheorem{remark}[theorem]{Remark}

\numberwithin{equation}{section}

\title[Auslander-Reiten conjecture and complete intersection dimension]{Auslander-Reiten conjecture for modules whose (self) dual has finite complete intersection dimension}

\author[D.~Ghosh]{Dipankar Ghosh}
\address{Department of Mathematics, Indian Institute of Technology Kharagpur, West Bengal - 721302, India}
\email{dipankar@maths.iitkgp.ac.in, dipug23@gmail.com}
\urladdr{\url{https://orcid.org/0000-0002-3773-4003}}

\author[M.~Samanta]{Mouma Samanta}
\address{Department of Mathematics, Indian Institute of Technology Kharagpur, West Bengal - 721302, India}
\email{mouma17@kgpian.iitkgp.ac.in, samantastm17@gmail.com}

%; Secondary 13D07, 13B22

\begin{document}

\pagenumbering{arabic}
\thispagestyle{empty}
\begin{abstract}
Over a commutative Noetherian ring, we show that the Auslander-Reiten conjecture holds true for the class of (finitely generated) modules whose dual has finite complete intersection dimension. We provide another result that validates the conjecture for the class of modules whose self dual has finite complete intersection dimension and either the module or its dual has finite Gorenstein dimension. Thus we combine and strengthen a number of results in the literature, due to Auslander-Ding-Solberg, Dey-Ghosh and Rubio-P\'{e}rez.
\end{abstract}
% \begin{abstract}
% For a finitely generated module $M$ over a commutative Noetherian ring $R$, it is proved that the CI-dimension (i.e., complete intersection dimension) of $M$ is zero \iff its dual
% % (i.e., $\Hom_R(M,R)$) 
% has finite CI-dimension and $\Ext_R^{i\ge 1}(M,R)=0$.
% %\com{Usually, we write a paper in present tense. So it should ``As a consequence, we show that... When we talk about what others already have done, we usually write it in past tense. For example, Araya-Yoshino proved that...}
% As a consequence, we show that Auslander-Reiten conjecture holds true for the class of modules whose dual has finite CI-dimension. Our another result validates Auslander-Reiten conjecture for the class of modules whose self dual has finite CI-dimension and either the module or its dual has finite Gorenstein dimension.  
% \end{abstract}
\maketitle

\section{Introduction}
% \begin{setup}
    Throughout the article, all rings are assumed to be commutative Noetherian, and all modules are assumed to be finitely generated. Let $R$ be a ring.
% \end{setup}

This article is inspired by one of the most celebrated and long-standing conjectures in homological commutative algebra, which is due to Auslander-Reiten \cite{AR75}.
% This is originated from representation theory of Artin algebras. It is known to be equivalent to the Generalized Nakayama Conjecture.

\begin{conjecture}[Auslander-Reiten]\label{conj:ARC}
{~}\\%\cite{AR75}
    An $R$-module $M$ is projective if $\Ext_R^i(M,M\oplus R)=0$ for all $i > 0$.
\end{conjecture}

\begin{para}\label{survey}
    Conjecture~\ref{conj:ARC} is known to hold true in the following cases: (1)~$R$ is complete intersection, due to Auslander-Ding-Solberg \cite[1.9]{ADS93}. (2)~$R$ is normal, due to Kimura-Otake-Takahashi \cite[3.13.(3)]{KOT22} under the assumption that $R$ is Cohen-Macaulay, and \cite[1.2]{Kim23} in general by Kimura. (3)~$M$ has finite complete intersection dimension (in short, CI-dimension), a result by Araya-Yoshino \cite[4.3]{AY98}. (4)~$\Hom_R(M,M)$ or $M^*:=\Hom_R(M,R)$ has finite injective dimension, due to Ghosh-Takahashi \cite[1.3]{GT24}. (5)~$M^*$ has finite projective dimension, by Dey-Ghosh \cite[6.9.(1)]{DG}. (6)~$M$ has finite Gorenstein dimension, and $\Hom_R(M,M)$ has finite projective dimension \cite[6.9.(2)]{DG}. (7)~$M^*$ has finite Gorenstein dimension and $\Hom_R(M,M)$ has finite projective dimension, due to Rubio-P\'{e}rez \cite[5.12]{RP24}. However, the conjecture is broadly open even for Gorenstein local rings. The main aim of this article is to combine and strengthen most of these results to CI-dimension.
\end{para}

The notion of CI-dimension is introduced by Avramov-Gasharov-Peeva in \cite{AGP97}. It is well known that the class of modules having finite CI-dimension contains all modules of finite projective dimension and all modules over complete intersection rings, cf.~\ref{para:G-CI-proj-dim} and \ref{AGP-1.3}. Motivated by this, we prove \Cref{conj:ARC} when $M^*$ has finite CI-dimension, which unifies and extends the results (1) and (5) in \ref{survey}.
%\cite[1.9]{ADS93}, \cite[Thm.~6.7]{DG23}, \cite[Thm.~4.1]{HSV04}, \cite[Thm.~0.1]{HL04}, \cite[Cor.~4]{Ar09}, \cite[Cor.~1.3]{KOT21}, \cite[Cor.~1.2]{Kim23}, \cite[Thm.~1.2]{Kum23}, \cite[1.2]{NS17}, \cite[Cor.~6.8]{NT20}, \cite[Prop.~5.10 and Thm.~6.7]{Tak23} 

\begin{customtheorem}{\ref{thm:les-ARC-finite-CI-dimension-M^*}}\label{thm-1st-main}
    Let $R$ be a local ring of depth $t$, and $M$ be an $R$-module such that
    % \begin{center}
        $\Ext_R^{1\le i \le t}(M,R) = 0$, and $\Ext_R^{2j}(M,M) = 0$ for some $j \ge 1$.
    % \end{center}
    Suppose
    % at least one of $\cidim_R(M)$ and 
    $\cidim_R(M^*)$ is finite. Then $M$ is free.
\end{customtheorem}
% {\new{Note that, this theorem does not follow from the result of Araya-Yoshino \cite[Theorem~4.3]{AY98}. Moreover, it is not necessarily true that $\cidim_R(M)$ is finite if and only if $\cidim_R(M^*)$ is finite, see, \Cref{ex:CI-dim-zero-pd-infinite}.}}

In \cite[Theorem~4.3]{AY98}, Araya-Yoshino proved that for a module $M$ over a local ring $R$, if $\cidim_R(M)$ is finite and $\Ext_R^j(M,M) = 0$ for all $j \ge 1$, then $M$ is free. \Cref{thm:les-ARC-finite-CI-dimension-M^*} provides a variation of the result of Araya-Yoshino. Note that \Cref{thm-1st-main} does not follow from \cite[4.3]{AY98}. Trivial examples can be constructed for which $\cidim_R(M)$ is possibly infinite even when $\cidim_R(M^*)=0$, cf.~\Cref{exam:ci-dim-M*-M}.

Next, we look at the case when $\Hom_R(M, M)$ has finite CI-dimension, and either $M$ or $M^*$ has finite Gorenstein dimension. We have the subsequent outcome, which confirms that \Cref{conj:ARC} holds true in this situation as well. This combines and considerably strengthens results (1), (6) and (7) in \ref{survey}.
% and hence, provides a further condition for a module to be free.
There are examples of modules $M$ over a local ring $R$ (which is not necessarily a complete intersection ring)
such that $\cidim_R(M)$, $\cidim_R(M^*)$ and $\cidim_R(\Hom_R(M, M))$ are finite, but both projective and injective dimensions of $M$, $M^*$ and $\Hom_R(M, M)$ are not finite; see \Cref{ex:CI-dim-zero-pd-infinite}. Thus, the present paper achieves something newer than that of the previous papers, such as \cite[1.3]{GT24}, \cite[6.9]{DG} and \cite[5.12]{RP24}. 

\begin{customtheorem}{\ref{thm:les-CI-dim-Hom(M,M)-G-dim-M}}
    Let $R$ be a local ring of depth $t$, and $M$ be an $R$-module such that $\Ext_R^{1\le i\le t}(M,R)=0$, $\Ext_R^{j \ge 1}(M,M)=0$ and $\cidim_R(\Hom_R(M,M))$ is finite. Suppose at least one of $\gdim_R(M)$ and $\gdim_R(M^*)$ is finite. Then $M$ is free.
\end{customtheorem}

As consequences of Theorems~\ref{thm:les-ARC-finite-CI-dimension-M^*} and \ref{thm:les-CI-dim-Hom(M,M)-G-dim-M}, we obtain the following.

\begin{corollary}\label{cor:ARC}
    Over a commutative Noetherian ring $R$, the Auslander-Reiten conjecture holds true for a module $M$ in each of the following cases:
    \begin{enumerate}[\rm (1)]
        \item $\cidim_R(M^*)$ is finite.
        % \item {\rm (Araya-Yoshino)} $\cidim_R(M)$ is finite.
        \item $\cidim_R(\Hom_R(M,M))$ and $\gdim_R(M)$ are finite.
        \item $\cidim_R(\Hom_R(M,M))$ and $\gdim_R(M^*)$ are finite.
    \end{enumerate}
    % \\
    % {\rm (1)} $\cidim_R(M^*)$ is finite.\\
    % {\rm (2)} {\rm (Araya-Yoshino)} $\cidim_R(M)$ is finite.\\
    % {\rm (3)} $\cidim_R(\Hom_R(M,M))$ and $\gdim_R(M)$ are finite.\\
    % {\rm (4)} $\cidim_R(\Hom_R(M,M))$ and $\gdim_R(M^*)$ are finite.
    % \begin{enumerate}[\rm (1)]
    %     \item $\cidim_R(M^*)<\infty$.
    %     \item $\cidim_R(\Hom_R(M,M))<\infty$ and $\gdim_R(M)<\infty$.
    %     \item $\cidim_R(\Hom_R(M,M))<\infty$ and $\gdim_R(M^*)<\infty$.
    % \end{enumerate}
\end{corollary}
% In Section~\ref{Preliminaries}, we recall some notations, terminologies, definitions and basic
% properties which are required for this article. In Section~\ref{Main results}, we present our main results on freeness criteria of a module. Along with that we provide some lemmas which are used to prove those results.
% In section 3 we prove
%Theorem 1.1 and explore some of consequences. In section 4, we prove Theorem 1.3,
%and in Section 5 we prove the freeness criteria given in Theorem 1.6 and 1.7

\section{Preliminaries}\label{Preliminaries}
%\com{Recall the basic setup, like $R$, $M$ and $N$. Add notations, terminologies, definitions and basic properties that will be used throughout the paper. For example, mention what is the meaning of $\Ext^{m \le i\le n}_R(M,R)=0$ for fixed integers $m$ and $n$. Recall CI-dimension, G-dimension and some of their properties that will be used. See my paper with Aniruddha.}\\
In this section, we recall some terminologies, definitions and basic properties which are used throughout the article. Let $M$ and $N$ be $R$-modules. For integers $m$ and $n$, by $\Ext^{m \le i\le n}_R(M,N)=0$, we mean that $\Ext^i_R(M,N)=0$ for all $m \le i \le n$. Set $M^*:=\Hom_R(M,R)$,
% (except in the  proof of \Cref{prop:les-proj-dim-finite}, where ), 
and $M^{**}:=(M^*)^*$.

\begin{para}\label{para:M-*-syzygy-module-of-M}
    Let $R$ be local. For $n\ge 0$, $\Omega^n_R(M)$ denotes the $n$th syzygy module of $M$ in a minimal free resolution of $M$, i.e., if $\cdots \rightarrow F_2 \xlongrightarrow{d_2}F_1\xlongrightarrow{d_1}F_0 \rightarrow 0$ is a minimal free resolution of $M$, then $\Omega^n_R(M) := \coker d_{n+1} $.
\end{para}
% \old{
% \begin{para}\label{para:syzygy-module}
%     An $R$-module $M$ is called an $n$-th syzygy if there exists an exact sequence $$0\rightarrow M\rightarrow P_{n-1} \rightarrow P_{n-2}\rightarrow \cdots \rightarrow P_0,$$ where each $P_i$ is a projective $R$-module.
% \end{para}
% }
\begin{para}\label{para:quasi-deformation-CI-dimension}
    \cite[(1.1) and (1.2)]{AGP97} Let $R$ be a local ring, and $M$ be an $R$-module. A (codimension $c$) quasi-deformation of $R$ is a diagram of local ring homomorphisms $ R\rightarrow R' \leftarrow S$, where $R\rightarrow R' $ is flat, and $R' \leftarrow S$ is a (codimension $c$) deformation, i.e., $ R' \leftarrow S $ is surjective with kernel generated by a (length $c$) regular sequence. When $M \neq 0$, CI-dimension of $M$ is defined to be $$\cidim_R(M)=\inf\{\pd_S(M')-\pd_S(R') : R\rightarrow R'\leftarrow S \mbox{ is a quasi-deformation} \},$$
    where $M' := M \otimes_R R' $. Set $\cidim_R(0):=0$. For a module $M$ over a Noetherian ring $R$ (not necessarily local), CI-dimension of $M$ is defined by
    $$\cidim_R(M)=\sup\{\cidim_{R_{\mathfrak m}}(M_{\mathfrak m}) : \mathfrak m \in \Max(R)\}.$$
\end{para}

The notion of Gorenstein dimension (in short, G-dimension) is due to Auslander, and it is developed by Auslander and Bridger in \cite{AB69}.

\begin{para}%\cite{AB69}
    An $R$-module $G$ is said to have Gorenstein dimension zero, 
    % i.e., \[ \gdim_R(G)=0, \] 
    if $G$ is reflexive, and $\Ext^i_R(G, R) = \Ext^i_R(G^*, R) = 0$ for all $i \ge 1$. We write it as $\gdim_R(G)=0$. For an $R$-module $M$, $\gdim_R(M)$ is defined to be the infimum of all integers $n \ge 0$ such that there exists an exact sequence $0\rightarrow G_{n} \rightarrow G_{n-1}\rightarrow \cdots \rightarrow G_0\rightarrow M\rightarrow 0$, where $\gdim_R(G_i) = 0$ for all $0 \le i \le n$. If no such $n$ exists, then $\gdim_R(M)=\infty$.
\end{para}

\begin{para}\label{para:G-dim-finite}
    (\cite[Theorem~2.3.16]{Ch00}) For a non-zero $R$-module $M$, if $\gdim_R(M)$ is finite, then $\gdim_R(M)=\sup\{n : \Ext^n_R(M,R)\neq 0\}$.
    % see, e.g., \cite[Theorem~2.3.16]{Ch00}.
\end{para}

\begin{para}\label{para:G-CI-proj-dim}
    For an $R$-module $M$, it is shown in \cite[1.4]{AGP97} that there are inequalities:
    $$\gdim_R(M)\le \cidim_R(M)\le \pd_R(M).$$ If one of these is finite, then it is equal to those to its left. Moreover, when $R$ is local, and any of these dimensions is finite, then it is equal to $\depth(R)-\depth(M)$, see \cite[(4.13.b)]{AB69}, \cite[1.4]{AGP97} and \cite[1.3.3]{BH98} respectively.
\end{para}

\begin{para}\label{AGP-1.3}
    (\cite[(1.3)]{AGP97} and \cite[(4.20)]{AB69}) For a local ring $(R,\fm,k)$,
    the following are equivalent:
		\begin{enumerate}[(1)]
			\item $R$ is complete intersection (resp., Gorenstein).
			\item $\cid_R(M) < \infty$ (resp., $\gdim_R(M) < \infty$) for every $R$-module $M$.
			\item $\cid_R(k) < \infty$ (resp., $\gdim_R(k) < \infty$).
		\end{enumerate}
\end{para}

Using a standard change of rings spectral sequence, one obtains the following.

\begin{lemma}\cite[Theorem~10.75]{Rot09}\label{lem:les-Ext-S-Ext-R}
Set $R := S/(x)$, where $x$ is an $S$-regular element. Let $M$ and $N$ be $R$-modules. Then, we have the following long exact sequence:
    \begin{equation*}
        \begin{split}
            0 \longrightarrow &\Ext^1_R(M,N)\longrightarrow \Ext^1_S(M,N) \longrightarrow \Ext^0_R(M,N)\longrightarrow \\
           & \hspace{4cm}\vdots\\
           & \Ext^i_R(M,N)\longrightarrow \Ext^i_S(M,N) \longrightarrow \Ext^{i-1}_R(M,N)\longrightarrow\\
           & \Ext^{i+1}_R(M,N)\longrightarrow \Ext^{i+1}_S(M,N) \longrightarrow \Ext^i_R(M,N)\longrightarrow \cdots.
        \end{split}
    \end{equation*}
\end{lemma}

\section{Main results}\label{Main results}

%\com{Comparing Lemmas~2.1 and 2.2, I found inconsistency in the notations. In 2.1, $S$ is quotient of $R$, while in 2.2, $R$ is quotient of $S$. You should be consistent with the notations throughout the paper.}
%\com{I have used ``rm" below to remove the italic environment from the numbering. See the difference between {\rm (1)} and {\it (1)}.}

Here we prove Theorems~\ref{thm:les-ARC-finite-CI-dimension-M^*} and \ref{thm:les-CI-dim-Hom(M,M)-G-dim-M}. For that, first we prepare some lemmas.

\begin{lemma}\label{lem:les-vanishing-Ext-R-Ext-S}
    Let $S$ be a local ring. Set $R := S/(\textbf{x})$, where $ \textbf{x} := x_1,\dots, x_c$ is an $S$-regular sequence. Let $M$ be an $R$-module such that $\Ext^{m \le i\le n}_R(M,R)=0$ for some integers $m,n\ge 1$. Then $\Ext^{m+c \le i\le n}_S(M,S)=0$.
\end{lemma}

\begin{proof}
    %\com{Use induction on $c$, and prove the lemma.}\\
    We proceed by induction on $c$. When $c=0$, there is nothing to show.
    
    Let $c=1$. Then $R=S/(x_1)$. Since $\Ext^{m \le i\le n}_R(M,R)=0$, in view of Lemma~\ref{lem:les-Ext-S-Ext-R}, $\Ext^{m+1 \le i\le n}_S(M,R)=0$. Now, the short exact sequence
    $0\rightarrow S \xrightarrow{x_1} S \rightarrow R \rightarrow 0$
    induces an exact sequence  $\Ext^i_S(M,S)\xrightarrow{x_1} \Ext^i_S(M,S) \rightarrow \Ext^i_S(M,R)$
    for each $i\ge 0$. So, by Nakayama's lemma, it follows that $\Ext^{m+1 \le i\le n}_S(M,S)=0$.
    
    Let $c\ge 2$. Set $R' := S/(x_1,\dots,x_{c-1})$. Then $x_c$ is $R'$-regular, and $R$ is isomorphic to $R'/(x_c)$. Therefore, since $\Ext^{m \le i\le n}_R(M,R)=0$, by the previous case, one gets that
    % in view of Lemma~\ref{lem:les-Ext-S-Ext-R}, one has that $\Ext^{m+1 \le i\le n}_{R'}(M,R)=0$. Now from the short exact sequence $0\rightarrow R' \xrightarrow{x_c} R' \rightarrow R \rightarrow 0$, we get the exact sequences $\Ext^i_{R'}(M, R')\xrightarrow{x_c} \Ext^i_{R'}(M, R') \rightarrow \Ext^i_{R'}(M,R)$ for all $i\ge0$. Consequently, by Nakayama's lemma, it follows that
    $\Ext^{m+1 \le i\le n}_{R'}(M,R')=0$. So the claim follows by induction hypothesis.
    % So, by the induction hypothesis, we have that $\Ext^{m+c \le i\le n}_S(M,S)=\Ext_S^{(m+1)+(c-1)\le i\le n}(M,S) = 0$.  
\end{proof}

\begin{lemma}\label{lem:les-Ext-S-iso-Hom-R}
Let $S$ be a local ring. Set $R := S/(\textbf{x})$, where $ \textbf{x} := x_1,\dots, x_c$ is an $S$-regular sequence. Let $M$ be an $R$-module such that $\Ext^{1 \le i\le c+n}_R(M,R)=0$ for some integer $n \ge 1$. Then 
\begin{enumerate}[\rm (1)]
\item $\Ext^i_S(M,S)=0$ for all $i \in \{0,\dots,c+n\} \smallsetminus \{c\} $, and
\item $\Ext^c_S(M,S) \cong \Hom_R(M,R)$.
\end{enumerate}
\end{lemma}

\begin{proof}
   (1) Since $(x_1,\dots,x_c) \subset \ann_S(M)$, one has that $\depth(\ann_S(M),S) \ge c$. Consequently, $\Ext^{0\le i\le c-1}_S(M,S)=0$. As $\Ext^{1 \le i\le c+n}_R(M,R)=0$, by \Cref{lem:les-vanishing-Ext-R-Ext-S}, it follows that $\Ext^{c+1 \le i\le c+n}_S(M,S)=0$.
    
    (2) In view of \cite[Lemma~1.2.4]{BH98}, $\Ext_S^c(M,S) \cong \Hom_S(M,R) = \Hom_R(M,R)$.
\end{proof}

The following result is crucial in order to compare the finiteness of $\cidim_R(M)$ and $\cidim_R(M^*)$ in \Cref{prop:ci-dim-M*-M}.

\begin{proposition}\label{prop:les-proj-dim-finite}
    Let $R$ be a local ring of depth $t$, and $M$ be a non-zero $R$-module. Assume that $\Ext^i_R(M,R)=0$ for all $i\in \{0,\dots,c+t\} \smallsetminus \{c\} $, and $\pd_R(\Ext^c_R(M,R))$ is finite, for some integer $c\ge 0$. Then $\pd_R(M)=c$.
\end{proposition}

\begin{proof}
    When $c=0$, the result follows from \cite[Proposition~3.15.(2)]{DG}. So we may assume that $c\ge 1$. Consider a minimal free resolution of $M$ (along with the augmentation) \[ \cdots \longrightarrow F_2 \xlongrightarrow{d_2} F_1\xlongrightarrow{d_1} F_0 \xlongrightarrow{d_0} M \longrightarrow 0. \]
    Dualizing it, one gets a cochain complex
    $$ 0\longrightarrow M^* \xlongrightarrow{d_0^*} F_0^* \xlongrightarrow{d_1^*} F_1^* \longrightarrow \cdots \longrightarrow F_{c-1}^* \xlongrightarrow{d_c^*} F_c^* \xlongrightarrow{d_{c+1}^*} F_{c+1}^* \xlongrightarrow{d_{c+2}^*} \cdots.$$
    By the given conditions, the above complex gives us the following exact sequences:
    %\com{Note that $M^*\neq 0$ when $M\neq0$ and $c=0$. So you cannot just put $M^*=0$. You may have to consider the case $c=0$ separately. Then you may assume that $c\ge 1$, and give this proof.}
    \begin{align}
        &0\longrightarrow F_0^* \xlongrightarrow{d_1^*}F_1^* \xlongrightarrow{d_2^*}\cdots\xlongrightarrow{d_{c-1}^*}F_{c-1}^*\longrightarrow \Image d_c^* \longrightarrow 0 \label{eq:pd-image(d_c^*)}\\
        &0 \longrightarrow \Image d_c^* \longrightarrow \Ker d_{c+1}^* \longrightarrow \Ext^c_R(M,R) \longrightarrow 0 \label{eq:Image-Ker-Ext}\\
        &0 \longrightarrow \Image d_{c+1}^* \longrightarrow F_{c+1}^* \xlongrightarrow{d_{c+2}^*} F_{c+2}^* \longrightarrow \cdots \longrightarrow F_{c+t}^* \xlongrightarrow{d_{c+t+1}^*} F_{c+t+1}^*. \label{eq:Image(d_c^*)-syzygy-mod}
    \end{align}
    In view of \eqref{eq:pd-image(d_c^*)}, $\pd_R(\Image d_c^*)<\infty$. Thus, since $\pd_R(\Ext^c_R(M,R)) < \infty$, the exact sequence \eqref{eq:Image-Ker-Ext} yields that $\pd_R(\Ker d_{c+1}^*)<\infty$. Consequently, as $F^*_c/\Ker d^*_{c+1}$ is isomorphic to $\Image d^*_{c+1}$, it follows that $\pd_R(\Image d^*_{c+1})< \infty$. Now from the sequence \eqref{eq:Image(d_c^*)-syzygy-mod}, one obtains that $\Image d_{c+1}^*$ is a $(t+1)$st syzygy module. So $\depth(\Image d_{c+1}^*) \ge \min\{t+1,\depth(R)\} = \depth(R)$. Therefore, by the Auslander-Buchsbaum formula, $\pd_R(\Image d_{c+1}^*)=0$, i.e., $\Image d_{c+1}^*$ is free. As a result, the short exact sequence $0\rightarrow \Ker d_{c+1}^* \rightarrow F_c^*\xlongrightarrow{d_{c+1}^*} \Image d_{c+1}^* \rightarrow 0 $ splits, and hence $\Ker d_{c+1}^*$ is also free. Now dualizing the exact sequence 
    \[ F_{c+1} \xlongrightarrow{d_{c+1}} F_c \rightarrow \Image d_c \rightarrow 0,\]
    there is another exact sequence $0 \rightarrow (\Image d_c)^* \rightarrow F_c^* \xlongrightarrow{d_{c+1}^*} F_{c+1}^*$, which induces that $(\Image d_c)^*\cong \Ker d_{c+1}^*$. Thus, since $\Ker d_{c+1}^*$ is free, $(\Image d_c)^*$ is also free. Note that $\Ext^j_R(\Image d_c,R) = \Ext^j_R(\Omega_R^c(M),R) \cong \Ext^{j+c}_R(M,R)=0$ for all $1\le j\le t$. Therefore, by \cite[Lemma~3.3]{DEL21} or \cite[Proposition~3.15.(2)]{DG}, we have $\Omega_R^c(M)=\Image d_c$ is free. So $\pd_R(M) \le c$. Since $M$ is non-zero, $\ann_R(M)$ is a proper ideal of $R$, and hence $0\le \depth(\ann_R(M),R)\le t$. Note that $\depth(\ann_R(M),R) = \inf\{i:\Ext^i_R(M,R)\neq 0\}$. Therefore, since $\Ext^i_R(M,R)=0$ for all $i\in \{0,\dots,c+t\} \smallsetminus \{c\} $, it follows that $\Ext^c_R(M,R)\neq 0$. Hence $\pd_R(M) \ge c$. Thus $\pd_R(M) = c$. 
\end{proof}

%\begin{lemma}\label{Lemma3}
   % Let $S$ be a local ring. Set $R := \frac{S}{(\textbf{x})}$, where $ \textbf{x} := x_1,\dots, x_c$ is an $S$-regular sequence. Let $M$ be an $R$-module such that $\pd_S(M)<\infty$, and $\Ext^{i\geq 1}_R(M,M)= 0 $. Then, $M$ is free $R$-module.
%\end{lemma}
%\begin{proof}
     %By the given hypothesis, we have a quasi-deformation $R\xlongrightarrow{\id_R}R\longleftarrow S$. As we have $\pd_S(M)<\infty$ and $\pd_S(R)<\infty$, from the definition of complete intersection dimension, it follows that $\cidim_R(M)< \infty$. Since $ \Ext^{i\geq 1}_R(M,M)= 0 $, the claim of $M$ being $R$ free follows from \cite[Theorem~4.3]{AY98}.
%\end{proof}
%\com{Replace \Cref{prop:ci-dim-M*-M} by the following one.}

In the following proposition, the implication (1) $\Rightarrow$ (3) is due to Bergh-Jorgensen \cite[Lemma~3.5]{BJ11}. Our contribution here is the implication (4) $\Rightarrow$ (1). As a consequence, we also recover (1) $\Rightarrow$ (3).

\begin{proposition}\label{prop:ci-dim-M*-M}
    Let $R$ be a local ring of depth $t$, and $M$ be an $R$-module. Then the following are equivalent:
    \begin{enumerate}[\rm (1)]
        \item $\cidim_R(M)=0$.
        \item $\cidim_R(M)< \infty$ and $\Ext_R^{1\le i \le t}(M,R) = 0$.
        \item $\cidim_R(M^*)=0$ and $\Ext_R^{1\le i \le t}(M,R) = 0$.
        \item $\cidim_R(M^*)< \infty$ and $\Ext_R^{1\le i \le t}(M,R) = 0$.
    \end{enumerate}
\end{proposition}

\begin{proof}
    (4) $\Rightarrow$ (1):
    We may assume that $M$ is non-zero. Since $\cidim_R(M^*)< \infty$, in view of \ref{para:G-CI-proj-dim}, $\gdim_R(M^*)< \infty$. Hence, by \cite[Theorem~4.1]{RP24}, $\gdim_R(M)=0$. So, in particular, $\Ext_R^{i\ge 1}(M,R) = 0$. Since $\cidim_R(M^*)< \infty$, there is a quasi-deformation $R\rightarrow R' \leftarrow S$ such that $\pd_S(\Hom_R(M,R) \otimes_{R} R')<\infty$. Note that $R' = S/(\textbf{x})$ for some $S$-regular sequence $\textbf{x}:=x_1,\dots, x_c$. Since $R\rightarrow R'$ is a flat ring homomorphism, it follows that $\Ext_{R'}^{i \ge 1}(M',R') = 0$, and $\pd_S(\Hom_{R'}(M',R'))<\infty$, where $M'$ denotes the $R'$-module $M \otimes_{R} R'$. Therefore, by \Cref{lem:les-Ext-S-iso-Hom-R}, one obtains that $\Ext^i_S(M',S)=0$ for all $i\neq c$, and $\pd_S(\Ext^c_S(M',S))<\infty$. Thus \Cref{prop:les-proj-dim-finite} results in $\pd_S(M')=c$. As $\pd_S(R')=c$, it follows that 
    \[ \pd_S(M')-\pd_S(R')=0,\]
    and hence $\cidim_R(M)=0$.

    (1) $\Leftrightarrow$ (2): In both the conditions, $\cidim_R(M)<\infty$. In this case, by \ref{para:G-CI-proj-dim} and \ref{para:G-dim-finite}, one has that
    $\cidim_R(M)=\gdim_R(M)=\sup\{n : \Ext^n_R(M,R)\neq 0\}\le \depth(R).$
    Hence $\Ext_R^{1\le i \le \depth(R)}(M,R) = 0$ \iff $\cidim_R(M)=0$.
    
    (1) $\Rightarrow$ (3):
    Since $\cidim_R(M)=0$, it follows that $\gdim_R(M)=0$. Therefore $M\cong M^{**}$, $\Ext_R^{i\ge 1}(M, R)=0$ and $\Ext_R^{i\ge 1}(M^*, R)=0$. Thus $\cidim_R(M^{**})=0$ and $\Ext_R^{i\ge 1}(M^*, R)=0$. So, by the implication (4) $\Rightarrow$ (1) above, $\cidim_R(M^*)=0$. 

    (3) $\Rightarrow$ (4): This implication is trivial.
\end{proof}

\begin{remark}\label{rmk:G-dim-0-equiv-conds}
    The counterpart of \Cref{prop:ci-dim-M*-M} for G-dimension holds true. Indeed, replacing CI-dimension by G-dimension in \Cref{prop:ci-dim-M*-M}, the forward implications are clear from the definition of G-dimension, while the implication (4) $\Rightarrow$ (1) is shown in \cite[Theorem~4.1]{RP24}.
\end{remark}

\begin{remark}
    The counterpart of \Cref{prop:ci-dim-M*-M} for projective dimension also holds true. The forward implications in this case are also trivial, and the implication (4)~$\Rightarrow$~(1) is proved in \cite[Proposition~3.15.(2)]{DG}.
\end{remark}

The following example ensures that the vanishing condition on $\Ext_R^{1\le i \le t}(M,R)$ cannot be removed or weakened from \Cref{prop:ci-dim-M*-M}.(3).

\begin{example}\label{exam:ci-dim-M*-M}
    Let $(R, \fm, k)$ be a local ring of depth $t\ge 2$. Suppose that $R$ is not a complete intersection ring. Then, $\Hom_R(k, R) = 0$ and $\Hom_R(\fm, R) \cong R$ both are free $R$-modules. However, none of $k$ and $\fm$ has finite CI-dimension. Note that $\Ext_R^{1\le i \le t-1}(k,R) = 0$ and $\Ext_R^t(k,R) \neq 0$.
    %$\Ext_R^{1\le i \le t-2}(\fm,R) = 0$ and $\Ext_R^{t-1}(\fm,R) \neq 0$.
\end{example}

Now we are in a position to prove the first main  result of this article.

\begin{theorem}\label{thm:les-ARC-finite-CI-dimension-M^*}
    Let $R$ be a local ring of depth $t$, and $M$ be an $R$-module such that
    % \begin{center}
        $\Ext_R^{1\le i \le t}(M,R) = 0$, and $\Ext_R^{2j}(M,M) = 0$ for some $j \ge 1$.
    % \end{center}
    Suppose
    % at least one of $\cidim_R(M)$ and 
    $\cidim_R(M^*)$ is finite. Then $M$ is free.
\end{theorem}

\begin{proof}
    By virtue of \Cref{prop:ci-dim-M*-M}, one obtains that $\cidim_R(M)=0$. Therefore \cite[Theorem~4.2]{AB00} yields that $\pd_R(M)$ is finite. Consequently, in view of \ref{para:G-CI-proj-dim}, $\pd_R(M)=\cidim_R(M)=0$, i.e., $M$ is free.
\end{proof}

%\section{ARC for the modules which has finite G-dim and whose self dual has finite CI-dim}

%\old{For an $R$-module $M$, the following conditions are given.
%\begin{enumerate}
   % \item $\cidim_R(\Hom_R(M,M))<\infty$.
   %\item $\Ext^{i\ge 1}_R(M,R)=0$.
   %\item $\Ext^{i\ge 1}_R(M,M)=0$.
   %\item $\gdim_R(M)<\infty$.
%\end{enumerate}
%The condition (1) implies that there exists a quasi-deformation $R\rightarrow R' \leftarrow S$ of codimension $c$ such that
%$$\pd_S(\Hom_{R'}(M',M'))< \infty, \mbox{ where } M'=M\otimes_R R'.$$
%As $R\rightarrow R'$ is a local flat homomorphism, we may replace $R'$ by $R$. Let $t=\depth R$. The conditions (2) and (4) ensure that $\gdim_R(M)=0$, and hence $M\cong \Omega(Y)\oplus F$, where $\depth_R(Y) \ge \depth R$ and $F$ is free $R$-module. Therefore $\depth_R(M)\ge t$. Note that $\depth_S(\Hom_R(M,M))=\depth_R(\Hom_R(M,M)) \ge \depth_R(M)$ by \cite[Lemma~3.1]{DG}. So we have $$\pd_S(\Hom_R(M,M))=\depth S - \depth_S(\Hom_R(M,M))\le \depth S - \depth_R(M) \le \depth S- t = c.$$
%Condition (2) implies that $\Ext^{i\neq c}_S(M,S)=0$ and $\Ext^c_S(M,S) \cong \Hom_R(M,R)$, by Lemma \ref{Lemma1}. Condition (3) implies that $\Ext^{i\ge c+1}_S(M,M)=0$. Since $\gdim_R(M)=0$, one also has that $\gdim_R(M^*)=0$. Thus $\gdim_S(M)=c=\gdim_S(M^*)$. Therefore $\gdim_S(\Ext^c_S(M,S))=c.$ This implies that $\Ext_S^c(\Ext^c_S(M,S),S) \neq 0$ and $\Ext_S^{i\ge c+1}(\Ext^c_S(M,S),S) =0$.}

%\old{We should consider $ \Omega_S^c(M) $, and try to show that it is free.}

In order to prove Theorem~\ref{thm:les-CI-dim-Hom(M,M)-G-dim-M}, we need the following elementary lemma.

\begin{lemma}\label{lem:les-Ext-S-Ext-S/(x)}
    Let $S$ be a local ring, and $x$ be an $S$-regular element. Let $L$ and $N$ be two $S$-modules such that $x$ is regular on both $L$ and $N$. For a fixed $i\ge 0$, suppose $\Ext^i_{S/(x)}(L/xL, N/xN)=0$. Then $\Ext^i_S(L,N)=0$. 
\end{lemma}
\begin{proof}
    Since $x$ is both $S$-regular and $L$-regular, by \cite[Lemma~2.(ii), p.~140]{Mat86}, one has that $\Ext^n_S(L, N/xN) \cong \Ext^n_{S/(x)}(L/xL, N/xN)$ for all $n\ge 0$. Therefore, by the given condition, $\Ext^i_S(L, N/xN)=0$. Now the short exact sequence 
    \[ 0\rightarrow N\xrightarrow{x} N \rightarrow N/xN\rightarrow 0 \] 
    yields another exact sequence 
    \[ \Ext_S^i(L,N)\xrightarrow{x} \Ext_S^i(L,N) \rightarrow \Ext_S^i(L, N/xN).\] 
    As $\Ext^i_S(L, N/xN)=0$, by Nakayama's lemma, it follows that $\Ext_S^i(L,N)=0$. 
\end{proof}

Here is the second main result of this article.

\begin{theorem}\label{thm:les-CI-dim-Hom(M,M)-G-dim-M}
    Let $R$ be a local ring of depth $t$, and $M$ be an $R$-module such that $\Ext_R^{1\le i\le t}(M,R)=0$, $\Ext_R^{j \ge 1}(M,M)=0$ and $\cidim_R(\Hom_R(M,M))$ is finite. Suppose at least one of $\gdim_R(M)$ and $\gdim_R(M^*)$ is finite. Then $M$ is free.
\end{theorem}

\begin{proof}
    If $\gdim_R(M^*)<\infty$, then by \cite[Theorem~4.1]{RP24}, $M$ is totally reflexive, i.e., $\gdim_R(M)=0$. So, in any case, we may assume that $\gdim_R(M)<\infty$. Then $\gdim_R(M) = \sup\{i : \Ext_R^i(M,R)\neq 0\}\le t$, cf.~\ref{para:G-dim-finite} and \ref{para:G-CI-proj-dim}. Therefore, since $\Ext_R^{1\le i\le t}(M,R)=0$, one concludes that $\gdim_R(M)=0$.
    %Consequently, $\Ext_R^{j\ge 1}(M,R)=0$.
    
    Since $\cidim_R(\Hom_R(M,M))< \infty$, by the definition of CI-dimension, there exists a quasi-deformation $R\rightarrow R' \leftarrow S$ such that $\pd_S(\Hom_R(M,M) \otimes_{R} R')$ is finite. Note that $R' = S/(\textbf{x})$ for some $S$-regular sequence $\textbf{x}:=x_1,\dots, x_c$, and $R\rightarrow R'$ is a local flat homomorphism. Set $M':=M\otimes_RR'$. Then $\pd_S(\Hom_{R'}(M',M'))$ is finite. Since $\Ext_R^{j \ge 1}(M,M)=0$, one has that $\Ext_{R'}^{j \ge 1}(M',M')=0$. Moreover, $\gdim_{R'}(M')=0$ as $\gdim_R(M)=0$. Tensoring a minimal $R$-free resolution of $M$ by $R'$, one gets a minimal $R'$-free resolution of $M'$. So $\pd_R(M)=\pd_{R'}(M')$. Thus, we may replace $R'$ by $R$, and without loss of generality assume that $R=S/(\textbf{x})$.

    We complete $R=S/(\textbf{x})$ with respect to the maximal ideal of $S$, and assume that both $R$ and $S$ are complete local rings.
    Since $\Ext^2_R(M,M)=0$, in view of \cite[Proposition~1.7]{ADS93}, the $R$-module $M$ is liftable to $S$, i.e., there exists an $S$-module $N$ such that $N/\textbf{x}N\cong M$ and $\Tor^S_{i\ge 1}(N,R)=0$ (see \cite[page~276]{ADS93}). Thus $\Tor^S_{i\ge 1}(N,S/(\textbf{x}))=0$. Hence, since $\textbf{x}$ is an $S$-regular sequence, it is also $N$-regular, see, e.g., \cite[Lemma~6.1]{DG23}. So, as $\Ext_R^{j \ge 1}(M,M)=0$, by repeatedly applying Lemma~\ref{lem:les-Ext-S-Ext-S/(x)}, one obtains that $\Ext_S^{j \ge 1}(N,N)=0$. Consequently, in view of \cite[Lemma~3.1.(1)]{DG}, it follows that $\textbf{x}$ is regular on $\Hom_S(N,N)$, and $\Hom_S(N,N)/\textbf{x}\Hom_S(N,N) \cong \Hom_{R}(M, M)$. Thus $\pd_S(\Hom_S(N,N))<\infty$ as $\pd_S(\Hom_{R}(M, M))<\infty$, see, e.g., \cite[1.3.6]{BH98}. Since $\textbf{x}$ is regular on both $S$ and $N$, by \cite[Corollary~4.31]{AB69}, $\gdim_S(N)=\gdim_R(M)=0$. Therefore \cite[Corollary~6.8]{DG} yields that $N$ is $S$-free. Hence $\pd_R(M)=\pd_S(N)=0$ (cf.~\cite[1.3.5]{BH98}). Thus $M$ is $R$-free.
\end{proof}

The following example shows that in Theorems~\ref{thm:les-ARC-finite-CI-dimension-M^*} and \ref{thm:les-CI-dim-Hom(M,M)-G-dim-M}, the condition on self Ext vanishing cannot be omitted.

\begin{example}\label{ex:CI-dim-zero-pd-infinite}
    Let $S=T[[x]]$ or $S=T[x]$ (a formal power series ring or a polynomial ring in a variable $x$) over a local ring $T$. Set $R:=S/(x^n)$, where $n\ge 2$. Consider the $R$-module $M:=Rx$. Then, the following hold.
    \begin{enumerate}[\rm (1)]
        \item $\cidim_R(M)=0$, $\cidim_R(M^*)=0$ and $\cidim_R(\Hom_R(M,M))=0$.
        \item $\Ext_R^{i \ge 1}(M,R) = 0$, and $\Ext_R^j(M,M)\neq 0$ for every $j\ge  0$.
        \item 
        $\pd_R(M)=\infty=\id_R(M)$, $\pd_R(\Hom_R(M,M))=\infty=\id_R(\Hom_R(M,M))$, and $\pd_R(M^*)=\infty=\id_R(M^*)$.
        \item The ring $R$ is Gorenstein \iff $T$ is so.
    \end{enumerate}
\end{example}

\begin{proof}
    (1) Note that $M\cong R/(x^{n-1})\cong S/(x^{n-1})$. Considering the trivial quasi-deformation $R\rightarrow R \leftarrow S$, since $\pd_S(M)-\pd_S(R) = 1-1 = 0$, it follows that $\cidim_R(M)=0$. Next, we notice that
    % Next we show that $M^*\cong M \cong \Hom_R(M,M)$,
    % Using this, the other two equalities follow from the isomorphisms below:
    \begin{align*}
        &M^* \cong \Hom_R(R/(x^{n-1}), R) \cong (0:_Rx^{n-1}) = Rx = M \quad\mbox{and}\\
        &\Hom_R(M,M)\cong \Hom_R(R/(x^{n-1}), M) \cong (0:_Mx^{n-1}) = M.
    \end{align*}
    Consequently, $\cidim_R(M^*)=0$ and $\cidim_R(\Hom_R(M,M))=0$.

    (2) Consider the minimal free resolution
    \begin{equation}\label{min-reso}
        \mathbb{F}_M: \quad\quad \cdots \rightarrow R \xrightarrow{x} R\xrightarrow{x^{n-1}} R\xrightarrow{x} R\xrightarrow{x^{n-1}} R \rightarrow 0
    \end{equation}
    of $M$ over $R$. It follows that $\pd_R(M)=\infty$. Dualizing \eqref{min-reso}, one computes that $\Ext_R^{i \ge 1}(M,R) = 0$. On the other hand, the cochain complex $\Hom_R(\mathbb{F}_M,M)$ is
    $$0 \rightarrow M \xrightarrow{x^{n-1}} M\xrightarrow{x}M\xrightarrow{x^{n-1}} M \xrightarrow{x} M \rightarrow \cdots.$$
    Computing its cohomologies, we observe that $\Ext_R^j(M,M)\neq 0$ for every $j\ge  0$ as
    \begin{equation*}
    \Ext_R^j(M,M)= 
    \begin{cases}
        Rx^{n-1}=x^{n-2}M & \mbox{$j\ge 1$ is odd}\\
        M/xM & \mbox{$j \ge 1$ is even.}%\cong Rx^{n-1}
    \end{cases}
    \end{equation*}
    
    (3) Since $\Ext_R^j(M,M)\neq 0$ for every $j\ge 0$, and $\Hom_R(M,M)\cong M\cong M^*$, the desired equalities follow.
    % (3) Since $\Ext_R^j(M,M)\neq 0$ for every $j\ge 0$, it follows that $\pd_R(M)=\infty=\id_R(M)$. Note that $\Hom_R(M,M)\cong M$. Therefore $\pd_R(\Hom_R(M,M))=\infty=\id_R(\Hom_R(M,M))$. 
    
    (4) Note that $S$ is local, and $x^n$ is an $S$-regular element. So $R$ is Gorenstein \iff $S$ is so, which is equivalent to that $T$ is Gorenstein.
\end{proof}

\begin{remark}
    Example~\ref{ex:CI-dim-zero-pd-infinite} also ensures that the answer to \cite[Question~6.2]{RP24} is negative, in general.
\end{remark}

Finally, we give the following.

\begin{proof}[Proof of Corollary~\ref{cor:ARC}]
    % Let $\Ext_R^i(M,M\oplus R)=0$ for all $i > 0$.
    In view of \cite[(4.15)]{AB69}, for every fixed $n\ge 0$, one has that $\gdim_R(M)\le n$ \iff $\gdim_{R_{\mathfrak m}}(M_{\mathfrak m})\le n$ for each $\mathfrak m \in \MaxSpec R$. On the other hand, by the definition of CI-dimension, $\cidim_{R_{\mathfrak m}}(M_{\mathfrak m}) \le \cidim_R(M)$ for all $\mathfrak m \in \MaxSpec R$. Note that $M$ is projective \iff $M_{\mathfrak m}$ is free for each $\mathfrak m \in \MaxSpec R$. Thus, replacing $R$ and $M$ by $R_{\fm}$ and $M_{\fm}$ respectively, we may assume that $R$ is local. Hence the corollary follows from Theorems~\ref{thm:les-ARC-finite-CI-dimension-M^*} and \ref{thm:les-CI-dim-Hom(M,M)-G-dim-M}.
\end{proof}
% \begin{corollary}\label{cor}
%     Over a commutative Noetherian ring $R$, the Auslander-Reiten conjecture holds true for a module $M$ in each of the following cases:
%     \begin{enumerate}[\rm (1)]
%         \item $\cidim_R(M^*)$ is finite.
%         \item {\rm (Araya-Yoshino)} $\cidim_R(M)$ is finite.
%         \item $\cidim_R(\Hom_R(M,M))$ and $\gdim_R(M)$ are finite.
%         \item $\cidim_R(\Hom_R(M,M))$ and $\gdim_R(M^*)$ are finite.
%     \end{enumerate}
% \end{corollary}

\subsection*{Acknowledgments}
The authors thank Souvik Dey for his suggestion on \Cref{exam:ci-dim-M*-M} and for his comments on the proof of Theorem~\ref{thm:les-CI-dim-Hom(M,M)-G-dim-M}. The authors also thank the anonymous referee for helpful suggestions. Mouma Samanta thanks the Government of India for financial support through the Prime Minister’s Research Fellowship for her Ph.D.

\end{document}